\documentclass[12pt,twoside]{article}
\usepackage[english]{babel}
\usepackage[latin1]{inputenc}
\usepackage[T1]{fontenc}

\usepackage[pdftex]{graphicx}
\usepackage{graphics}

\usepackage{amsmath}
\usepackage{amssymb}
\usepackage{amsthm}

\usepackage{array}
\usepackage{color}

\theoremstyle{plain}
\newtheorem{theo}{Theorem}[section]
\newtheorem{lem}[theo]{Lemma}
\newtheorem{prop}[theo]{Proposition}

\theoremstyle{definition}

\theoremstyle{remark}

\newcommand{\E}{\mathbb{E}}

\newcommand{\R}{\mathbb{R}}

\newcommand{\N}{\mathbb{N}}

\title{Adaptive Covariance Estimation with model selection} 

\author{Rolando Biscay, H\'el\`ene Lescornel and Jean-Michel Loubes}
\date{}

\begin{document}
\maketitle

\begin{abstract} 
We provide in this paper a fully adaptive penalized procedure to select a covariance among a collection of models observing i.i.d
replications of the process at fixed observation points. For this we generalize the results of ~\cite{MR2684389} and propose to use a data driven penalty to obtain an oracle inequality for the estimator. We prove that this method is an extension to the matricial regression model of the work by Baraud in~\cite{BAR}.
\end{abstract}
{\bf Keywords}: covariance estimation, model selection, adaptive procedure.
%



%

\section{Introduction}
Estimating the covariance function  of stochastic
processes is a fundamental issue in statistics with many applications, ranging from  geostatistics,
financial series or epidemiology for instance (we refer to \cite{MR1697409}, 
\cite{MR0456314} or \cite{MR1239641} for general references). While parametric methods have been extensively studied in the
statistical literature (see \cite{MR1239641} for a review),
nonparametric procedures have only recently received  attention, see for instance \cite{MR2403106,MR2684389,BIGOT:2010:HAL-00440424:4,Bigot:arXiv1010.1601} and references therein. \\
\indent In~\cite{MR2684389}, a model selection procedure is proposed to  construct a non parametric estimator of the covariance function of a stochastic process under mild assumptions. However their method  heavily relies on a prior knowledge of the variance. In this paper, we extend this procedure and propose a fully data driven penalty which leads to select the best covariance among a collection of models. This result constitutes a generalization to the matricial regression model of the selection methodology provided in~\cite{BAR}.\vskip .1in

Consider a stochastic process $\left( X\left( t \right) \right)_{t \in T}$ taking its values in $\R$ and indexed by $T \subset\R^d$, $d \in \N$. We assume that $\E \left[ X\left( t \right) \right]=0$ $\forall t \in T$ and we aim at estimating its covariance function $ \sigma \left( s, t \right)= \E \left[X \left( s\right) X\left( t \right) \right] < \infty$ for all $ t, s \in T$. We assume we observe  $X_i\left(t_j \right)$ where  $i\in \left\lbrace 1 \dots n \right\rbrace$ and $j \in \left\lbrace 1 \dots p\right\rbrace$. Note that the observation points $t_j$ are fixed and that the $X_i$'s are independent copies of the process $X$. Set  $x_i=\left(X_i\left(t_1\right),\dots , X_i\left( t_p\right) \right) \forall i \in \left\lbrace 1 \dots n \right\rbrace$ and denote by $\Sigma$ the  covariance matrix of $X$ at the observations points ${\Sigma =}\mathbb{E%
}\left( {x}_{i}{x}_{i}^{\top }\right) =\left( \sigma \left(
t_{j},t_{k}\right) \right) _{1\leq j\leq p,1\leq k\leq p}.$ \vskip .1in
Following the methodology presented in \cite{MR2684389}, we approximate the process $X$ by its projection onto some finite dimensional model. For this, consider a countable set of functions $\left(g_{\lambda}\right)_{\lambda \in \Lambda }$ which may be for instance a basis of $L^2 \left( T\right)$ and choose a collection of models $\mathcal{M} \subset \mathcal{P}\left( \Lambda \right)$. For  $m \subset \mathcal{M}$,  a finite number of indices, the process can be approximated by 
$$X\left( t\right)   \approx \sum_{\lambda \in m} a_{ \lambda} g_{\lambda}\left( t\right). $$      Such an approximation leads to an estimator which depends on the collection of functions $m$, denoted by $\hat{\Sigma}_m$. Our objective is to select in a data driven way, the best model, i.e. the one close to an oracle $m_0$ defined as the minimizer of the quadratic risk, namely $$m_0 \in \underset{m\in \mathcal{M}}{{\rm arg}\min}R\left(m\right) = \underset{m\in \mathcal{M}}{{\rm arg}\min} \E\left[ \left\Vert \Sigma - \hat{\Sigma}_m\right\Vert^ 2 \right].$$
This result is achieved using a model selection procedure.\vskip .1in
 The paper  falls into the
following parts. The description of the statistical framework of the matrix
regression is given in Section~\ref{s2}. Section \ref{s:covest} is devoted
to the main statistical results. Namely we recall the results of the estimate given in \cite{MR2684389} and prove an oracle inequality with a fully data driven penalty.  
Section~\ref{stech} states technical results which are used in all the
paper, while the proofs are postponed to the Appendix.

\section{Statistical model and notations} \label{s2}
We consider an $\R$-valued process $X\left(t \right)$ indexed by $T$ a subset of $\R^d$ with expectation equal to 0. We are interested in its covariance function denoted by $\sigma\left(s,t \right) = \E \left[ X\left(s\right) X\left( t \right) \right]$.

We have at hand the observations $x_i= \left( X_i \left( t_1 \right), \dots , X_i \left( t_p \right) \right)$ for $1 \leqslant i \leqslant n$ where  $X_i$ are independent copies of the process and $t_j$ are deterministic points.
We note $\Sigma \in \R^{p \times p}$ the covariance matrix of the vector $x_i$.

Hence we observe
\begin{equation}
\label{modmatr}
x_i x_i^ \top = \Sigma + U_i, \quad 1 \leqslant i \leqslant n
\end{equation}
where $U_i$ are i.i.d. error matrices with expectation 0. 
We denote by $S$ the empirical covariance of the sample : $S =\frac{1}{n} \sum_{i=1}^n x_i x_i^ \top$.

We use the Frobenius norm $\left\Vert \quad \right\Vert$ defined by
$\left\Vert A \right\Vert^ 2 = \mathrm{Tr}\left( A A^ \top\right)$ for all matrix $A$.
Recall that for a given matrix $A \in \mathbb{R}^ {p \times q}$, $vec(A)$ is the vector in $\in \mathbb{R}^ {pq}$ obtained by stacking the columns of $A$ on top of one another.
We denote by $A^-$ the reflexive generalized inverse of the matrix $A$, see for instance in~\cite{MR2365265} or \cite{engl96}.

The idea is to consider that we have a quite good approximation of the process in the following form 
\begin{equation}
X\left( t\right)   \approx \sum_{\lambda \in m} a_{ \lambda} g_{\lambda}\left( t\right), 
\end{equation}
where $m$ is a finite subset of a countable set $\Lambda$ , $\left( a_\lambda\right)  _{\lambda \in \Lambda}$ are random coefficients in $\R$ and $\left( g_\lambda\right) _{\lambda \in \Lambda}$ are real valued functions. We will consider models $m$ among  a finite collection denoted by $\mathcal{M}$ .\\

We note $G_m \in \mathbb{R}^{p \times \left\vert m \right\vert}$ where $\left( G_m\right)  _{j\lambda}= g_\lambda \left( t_j\right)  $ and $a_m$ the random vector of $\mathbb{R}^{|m|}$ with coefficients $\left( a_\lambda\right)  _{\lambda \in m}$.

Hence, we obtain the following approximations : 
$$x=\left( X\left( t_1\right)  ,..,X\left( t_p\right)  \right)  ^\top \approx G_{m}a_m$$
$$xx^\top \approx G_{m} a_m a_m^\top G_m^\top$$
$$\Sigma \approx G_{m}\mathbb{E}\left[    a_ma_m^\top\right]  G_m^\top$$

Thus, this point of view leads us to approximate $\Sigma$ by a matrix in the subset
\begin{equation}
\mathcal{S}\left( G_m \right)=\left\lbrace G_{m} \Psi G_m^\top / \Psi \text{ symmetric in } \mathbb{R}^{|m|\times |m|}\right\rbrace \subset \mathbb{R}^{p\times p}.
\end{equation}

Hence, for a model $m$, a natural estimator for  $\Sigma$ is given by the projection of $S$ onto $\mathcal{S}\left( G_m \right)$. We can prove using standard algebra (see in \cite{MR2684389} for a general proof) that it has the following form :
\begin{equation}
\widehat{\Sigma}_m= \Pi_m S \Pi_m \quad m \in \mathcal{M} \quad \in \R^{p\times p},
\end{equation}  
where 
\begin{align}
{\Pi }_{m}& ={G}_{m}\left( {G}_{m}^{\top }{G}%
_{m}\right) ^{-}{G}_{m}^{\top } \quad \in \R^{p \times p} \label{Pi-m}
\end{align}%
 are orthogonal projection matrices.
 Set
\begin{align}
D_{m}& =Tr\left( {\Pi }_{m}\otimes {\Pi }_{m}\right)   \notag
\end{align}%

which is the dimension of $\mathcal{S}\left(G_m\right)$ assumed to be positive, and $\Sigma_m=\Pi_m \Sigma \Pi_m$  the projection of $\Sigma$ onto this subspace.

Hence we obtain the model selection procedure defined in \cite{MR2684389}. The estimation error for a model $m\in \mathcal{M}$ is given by 
\begin{equation}
\mathbb{E}\left( \left\Vert {\Sigma }-\widehat{{\Sigma }}%
_{m}\right\Vert ^{2}\right) =\left\Vert {\Sigma -\Pi }_{m}{%
\Sigma \Pi }_{m}\right\Vert ^{2}+\frac{\delta _{m}^{2}D_{m}}{n},
\label{simple}
\end{equation}%
where%
\begin{align*}
\delta _{m}^{2}& =\frac{\mathrm{Tr}\left( \left( {\Pi }_{m}\otimes 
{\Pi }_{m}\right) {\Phi }\right) }{D_{m}}, \\
{\Phi }& {=}V\left( vec\left( {x}_{1}{x}%
_{1}^{\top}\right) \right) .
\end{align*}%

Given $\theta >0$, it is thus natural to define the penalized covariance estimator $\widehat{%
{\Sigma }}=\widehat{{\Sigma }}_{\widehat{m}}$ by 
\begin{equation*}
\widehat{m}=\arg \underset{m\in \mathcal{M}}{\min }\left\{ \frac{1}{n}%
\sum_{i=1}^{n}\left\Vert {x}_{i}{x}_{i}^{\top }-\widehat{%
{\Sigma }}_{m}\right\Vert ^{2}+pen\left( m\right) \right\} ,
\end{equation*}%
where 
\begin{equation}
pen\left( m\right) =\left( 1+\theta \right) \frac{\delta _{m}^{2}D_{m}}{n}.
\label{Penal}
\end{equation}
The following result proved in \cite{MR2684389} states an oracle inequality for the estimator $\widehat{\Sigma}$.

\begin{theo} \label{TeorConcentIneqLSEst} Let $q>0$ be given such
that there exists $%
\kappa >2\left( 1+q\right) $ satisfying $%
\mathbb{E}%
\left\Vert {x}_{1}%
{x}_{1}^{\top}\right\Vert ^{\kappa}<\infty $.
Then, for some constants $%
K\left( \theta \right) >1$ and $%
C^{\prime }\left( \theta ,\kappa,q\right) >0 $ we have that
\begin{equation*}%
\left( \mathbb{E}\left\Vert {\Sigma
}-\widehat{{\Sigma }} \right\Vert ^{2q}\right) ^{1/q}\leq
2^{\left( q^{-1}-1\right) _{+}} \left[ K\left( \theta \right)
\inf_{m\in \mathcal{M}}\left( \left\Vert
{%
\Sigma -\Pi }_{m}{\Sigma \Pi }_{m}\right\Vert ^{2}+%
\frac{%
\delta _{m}^{2}D_{m}}{n}\right) +\frac{\Delta _{\kappa}}{n}\delta _{\sup
}^{2} \right] ,
\end{equation*}%
where
\begin{equation*}%
\Delta
_{\kappa}^{q}=C^{\prime }\left( \theta ,\kappa,q\right) \mathbb{E}%
\left\Vert
{x%
}_{1}{x}_{1}^{\top}\right\Vert ^{\kappa}\left(
\sum\limits_{m\in
\mathcal{%
M}}\delta _{m}^{-\kappa}D_{m}^{-\left(
\kappa/2-1-q\right) }\right)
\end{equation*}%
and
\begin{equation*}
\delta
_{\sup }^{2}=\max \left\{ \delta
_{m}^{2}:m\in \mathcal{M}\right\} .
\end{%
equation*}%
\end{theo}


However the penalty defined here depends on the quantity $\delta_m$ which is unknown in practice since it relies on the matrix $\Phi =V\left( vec\left( x x^\top \right) \right)$.
Our objective is to study a covariance estimator built with a new penalty involving an estimator of $\Phi$.

More precisely, we will replace $pen(m)$ by an empirical version $\widehat{pen}(m)$, where

\begin{equation}
\widehat{pen}\left( m\right) =\left(1+\theta \right) \frac{\widehat{\delta} _{m}^{2}D_{m}}{n},
\label{Penalbis}
\end{equation}
and 
\begin{align*}
\widehat{\delta} _{m}^{2}& =\frac{\mathrm{Tr}\left( \left( {\Pi }_{m}\otimes 
{\Pi }_{m}\right) {\widehat{\Phi} }\right) }{D_{m}}, \\
\end{align*}%

with $\widehat{\Phi}$ an estimator of $\Phi$.

The objective is to generalize Theorem \ref{TeorConcentIneqLSEst} and to construct a fully adaptive penalized procedure to estimate the covariance function.

\section{Main result : adaptive penalized covariance estimation} \label{s:covest}
Here we state the oracle inequality obtained for the new covariance estimator introduced previously.

Set
$$y_i = vec \left( x_i x_i^\top \right), 1 \leqslant i \leqslant n,$$
which are vectors in $\R^{p^2}$ and denote  by $S_{vec} = \frac{1}{n}\sum_{i=1}^n y_i$ their empirical mean.
 Consider the following constant $C_{inf}=\inf_{m \in \mathcal{M}}\mathrm{Tr}\left( \left(\Pi_m \otimes \Pi_m \right)\Phi \right)$, and assume that the collection of models is chosen such that $C_{ inf} >0$. Set

\begin{align*}
\widehat{\Phi} & = \frac{1}{n} \sum_{i=1}^n \left( y_iy_i^\top -S_{vec}S_{vec}^\top \right),\\
\widehat{\delta} _{m}^{2}& =\frac{\mathrm{Tr}\left( \left( {\Pi }_{m}\otimes 
{\Pi }_{m}\right) {\widehat{\Phi} }\right) }{D_{m}}. \\
\end{align*}%
Given $\theta >0$, we consider the covariance estimator $\widetilde{%
{\Sigma }}=\widehat{{\Sigma }}_{\widetilde{m}}$ with 
\begin{equation*}
\widetilde{m}=\arg \underset{m\in \mathcal{M}}{\min }\left\{ \frac{1}{n}%
\sum_{i=1}^{n}\left\Vert {x}_{i}{x}_{i}^{\top }-\widehat{%
{\Sigma }}_{m}\right\Vert ^{2}+\widehat{pen}\left( m\right) \right\} ,
\end{equation*}%
where 
\begin{equation}
\widehat{pen}\left( m\right) =\left(1+\theta \right) \frac{\widehat{\delta} _{m}^{2}D_{m}}{n}.
\end{equation}

\begin{theo}

\label{theochili}

Let $1 \geqslant q>0$ be given such
that there exists $\beta > \max \left( 2\left( 1+2q\right), 3+2q\right)$ 
 satisfying $%
\mathbb{E}%
\left\Vert xx^\top\right\Vert ^{\beta}<\infty $.

Then, for a constant C depending on $\theta, \beta$ and $q$, we have for  $n\geqslant n(\beta,\theta,C_{inf},\Sigma)$, and $\forall \kappa \in  \left] 2\left( 1+2q\right) ; \min\left( \beta, 2 \beta -4\right) \right[$ : 

\begin{align}
\left( \E \left[ \left\Vert \Sigma - \widetilde{\Sigma} \right\Vert^{2q} \right] \right) ^{1/q} & \leqslant   C  \inf_{m\in \mathcal{M}}\left(\left\Vert
\Sigma -\Sigma_m\right\Vert ^{2}+%
\frac{\delta _{m}^{2}D_{m}}{n}\right) 
\\ + & \frac{C}{n}\left[ \widetilde{\Delta}_{\beta} \left[ \E \left[ 
\left\Vert xx^\top \right\Vert^\beta \right]^\frac{2}{\beta} + \left\Vert \Sigma \right\Vert^2 \right] + \delta^2_{sup} \Delta_\kappa \right]
\end{align}
where $$\widetilde{\Delta}^q_{\beta}=c\left(\theta, \beta, q\right)\left( \E \left[ \left\Vert x x^\top  \right\Vert^\beta \right]\sum_{m\in \mathcal{M}} \delta_m^{-\beta} D_m^{-\beta/2}\right)^ {1-\frac{2q}{\kappa}}$$ 

$$\Delta
_{\kappa}^{q}=C\left( \theta ,\kappa,q\right) \mathbb{E}%
\left\Vert
{x%
}{x}^{\top}\right\Vert ^{\kappa}\left(
\sum\limits_{m\in
\mathcal{%
M}}\delta _{m}^{-\kappa}D_{m}^{-\left(
\kappa/2-1-q\right) }\right)$$

and

$$\delta
_{\sup }^{2}=\max \left\{ \delta
_{m}^{2}:m\in \mathcal{M}\right\}.$$

\end{theo}
We have obtained in Theorem~\ref{theochili} an oracle
inequality since the estimator  $\widetilde{{\Sigma }}$ has the same quadratic
risk as the \textquotedblleft oracle\textquotedblright\ estimator except for
an additive term of order $O\left( \frac{1}{n}\right) $ and a constant
factor. Hence, the selection procedure is optimal in the sense that it
behaves as if the true model were at hand.

The proof of this theorem is divided into two parts. First, as in the of Theorem \ref{TeorConcentIneqLSEst} proved in \cite{MR2684389}, we will consider a vectorized version of the model \eqref{modmatr}. In this technical part we  will obtain an oracle inequality under some particular assumptions for a general penalty. In a second part, we will prove that our particular penalty verifies these assumptions by using properties of the estimator $\widehat{\Phi}$.

\section{Technical results} \label{stech}
\subsection{Vectorized model}
Here we consider the vectorized version of model \eqref{modmatr}. 
In this case, we observe the following vectors in $\R^{p^2}$ : 
\begin{equation}
\label{modvect}
y_i = f_i +\varepsilon_i \quad 1 \leqslant i \leqslant n.
\end{equation}
 Here $y_i$ corresponds  to $vec\left(x_ix_i^\top \right)$ in the model \eqref{modmatr}, $f_i$ to $vect\left(\Sigma \right)$ and  $\varepsilon_i$ to $vec\left(U_i\right)$.
We set $f=\left( f_1^\top , \dots, f_n^\top \right)^\top ,y=\left( y_1^\top , \dots, y_n^\top \right)^\top$ and $\varepsilon=\left( \varepsilon_1^\top , \dots, \varepsilon_n^\top \right)^\top$, which are vectors in $\R^{np^2}$.

We estimate $f$ by  an estimator of the form $$\widehat{f}_m = P_m y \quad m \in \mathcal{M},$$ where $P_m$ is the orthogonal projection onto a subspace $\mathcal{S}_m$ of dimension $D_m$. We note $f_m=P_mf$
and we consider the empirical norm $\left\Vert f \right\Vert^ 2_n =\frac{1}{n} \sum_{i=1}^n f_i^\top f_i$ with the corresponding scalar product $\left\langle \cdot, \cdot \right\rangle_n$.

First we state the vectorized form of Theorem \ref{TeorConcentIneqLSEst}.
Write 
\begin{align*}
\delta _{m}^{2}& =\frac{\mathrm{Tr}\left( {P}_{m}\left( {I}%
_{n}\otimes{\Phi }\right) \right) }{D_{m}}, \\
\delta _{\sup }^{2}& =\max \left\{ \delta _{m}^{2}:m\in \mathcal{M}\right\} .
\end{align*}%
Given $\theta >0$, define the penalized estimator $\widehat{{f}}=%
\widehat{{f}}_{\widehat{m}}$ , where%
\begin{equation*}
\widehat{m}=\arg \underset{m\in \mathcal{M}}{\min }\left\{ \left\Vert 
{y-}\widehat{{f}}_{m}\right\Vert _{n}^{2}+pen\left( m\right)
\right\} ,
\end{equation*}%
with%
\begin{equation*}
pen\left( m\right) =\left( 1+\theta \right) \frac{\delta _{m}^{2}D_{m}}{n}.
\end{equation*}

Then, the proof of Theorem \ref{TeorConcentIneqLSEst} relies on the following proposition proved in \cite{MR2684389}:
\begin{prop}
\label{ExtTh3.1Baraud}: Let $q>0$ be given such that there exists $\kappa>2\left(
1+q\right) $ satisfying $\mathbb{E}\left\Vert \mathbf{\varepsilon }%
_{1}\right\Vert ^{\kappa}<\infty $. Then, for some constants $K\left( \theta
\right) >1$ and $C\left( \theta ,\kappa,q\right) >0$ we have that%
\begin{equation}
\left( \mathbb{E}\left\Vert {f}-\widehat{{f}}\right\Vert
_{n}^{2q}\right) ^{1/q}\leq 2^{\left( q^{-1}-1\right) _{+}}\left[ K\left(
\theta \right) \inf_{m\in \mathcal{M}}\left( \left\Vert {f-P}_{m}%
{f}\right\Vert _{n}^{2}+\frac{\delta _{m}^{2}D_{m}}{n}\right) +\frac{%
\Delta _{\kappa}}{n}\delta _{\sup }^{2}\right] ,
\end{equation}%
where%
\begin{align*}
\Delta _{\kappa}^{q}& =C\left( \theta ,\kappa,q\right) \mathbb{E}\left\Vert \mathbf{%
\varepsilon }_{1}\right\Vert ^{\kappa}\left( \sum\limits_{m\in \mathcal{M}}\delta
_{m}^{-\kappa}D_{m}^{-\left( \kappa/2-1-q\right) }\right) .
\end{align*}
\end{prop}


The new estimator $\widetilde{\Sigma}$ defined previously corresponds here to the estimator  $\widetilde{{f}}=%
\widehat{{f}}_{\widetilde{m}}$ , where%
\begin{equation*}
\widetilde{m}=\arg \underset{m\in \mathcal{M}}{\min }\left\{ \left\Vert 
{y-}\widehat{{f}}_{m}\right\Vert _{n}^{2}+\widehat{pen}\left( m\right)
\right\} ,
\end{equation*}%
with%
\begin{equation*}
\widehat{pen}\left( m\right) =\left(1+ \theta \right) \frac{\widehat{\delta} _{m}^{2}D_{m}}{n},
\end{equation*}

and $\widehat{\delta} _{m}^{2}$ is some estimator of ${\delta} _{m}^{2}$.

Next Proposition gives an oracle inequality for this estimator under new assumptions on the model. As Proposition \ref{ExtTh3.1Baraud}, it is inspired by the paper \cite{BAR}.
\begin{prop}
\label{casgen}
Let $1 \geqslant q>0$ be given such
that there exists $%
\kappa>2\left( 1+2q\right) $ satisfying $%
\mathbb{E}%
\left\Vert \varepsilon _1\right\Vert ^{\kappa}<\infty $.

For $\alpha \in \left]0; 1 \right[$, set $\Omega = \cap_{m \in \mathcal{M}} \left\lbrace\widehat{\delta}_m^2 \geqslant \left(1 - \alpha \right)\delta_m^2 \right\rbrace$.

Assume that 
\begin{itemize}
\item[A1.] $\E \left[ \widehat{\delta}_m^2 \right] \leqslant \delta_m^2 $.
\item[A2.] $\mathbb{P} \left( \Omega^c \right) \leqslant\tilde{ C}\left(\alpha \right)\frac{1}{n^{\gamma}}$ for some $\gamma \geqslant \frac{q}{1- 2q/\kappa}$.
\end{itemize} 
Then, for a constant C depending on $\kappa, \theta$ and $q$, and we have

\begin{align}
\left( \E \left[ \left\Vert f - \widetilde{f} \right\Vert_n^{2q} \right] \right) ^{1/q} & \leqslant   C  \inf_{m\in \mathcal{M}}\left(\left\Vert
f -P _{m}f\right\Vert_n ^{2}+%
\frac{\delta _{m}^{2}D_{m}}{n}\right) 
\\ + & \frac{C}{n}\left[ \widetilde{\Delta}_\kappa \left[ \E \left[ 
\left\Vert \varepsilon_1 \right\Vert^\kappa \right]^\frac{2}{\kappa} + \left\Vert f \right\Vert_n^2 \right] + \delta^2_{sup} \Delta_\kappa \right]
\end{align}
where $$\widetilde{\Delta}^q_\kappa=\left( 
\tilde{C}\left(\alpha\right) \right)^ {\left(1-\frac{2q}{\kappa}\right)}\text{ with } \alpha = \alpha\left(\theta \right)\text{ is fixed in } \left]0;1\right[ $$ 
and $$ \Delta _{\kappa}^{q}=C\left( \theta ,\kappa,q\right) \mathbb{E}\left\Vert \mathbf{%
\varepsilon }_{1}\right\Vert ^{\kappa}\left( \sum\limits_{m\in \mathcal{M}}\delta
_{m}^{-\kappa}D_{m}^{-\left( \kappa/2-1-q\right) }\right) $$ 
\end{prop}

Theorem \ref{theochili} is thus a direct application of Proposition \ref{casgen}. Hence only remain to be checked the two assumptions A1 and A2.

\subsection{Auxiliary concentration type lemmas}
\label{auxlem}

Here we state some propositions required in the proofs of the previous results.

To our knowledge, the first is due to von Bahr and Esseen in \cite{MR0170407}.
\begin{lem}

\label{paraderose}
Let $U_1,\dots ,U_n$ independent centred variables with values in $\mathbb{R}$. For any $1 \leqslant \kappa \leqslant 2$ we have :
$$\E \left[\left\vert \sum_{i=1}^n U_i \right\vert^\kappa \right] \leqslant 8 \sum_{i=1}^n\E \left[\left\vert  U_i \right\vert^\kappa \right]$$

\end{lem}

The next proposition is proved in \cite{MR2684389}.
\begin{prop}
\label{proBaraud}
\label{ExtCor5.1Baraud}
Given $N,k\in \mathbb{N}$, let $\widetilde{{A}}\in \mathbb{R}%
^{Nk\times Nk}\diagdown \left\{ {0}\right\} $ be a non-negative
definite and symmetric matrix and ${\varepsilon }_{1},...,{%
\varepsilon }_{N}$ \ i.i.d random vectors in $\mathbb{R}^{k}$ with $\mathbb{E%
}\left( {\varepsilon }_{1}\right) =0$ and ${V}\left({%
\varepsilon }_{1}\right) ={\Phi }$. Write ${\varepsilon }%
=\left( {\varepsilon }_{1}^{\top},...,{\varepsilon }%
_{N}^{\top}\right) ^{\top}$, $\zeta \left( {\varepsilon }\right) =\sqrt{%
{\varepsilon }^{\top}\widetilde{A}{\varepsilon }}$, and $\delta
^{2}=\frac{\mathrm{Tr}\left( \widetilde{{A}}\left( {I}%
_{N}\otimes{\Phi }\right) \right) }{\mathrm{Tr}\left( \widetilde{%
{A}}\right) }$. For all $\beta\geq 2$ such that $\mathbb{E}\left\Vert
{\varepsilon }_{1}\right\Vert ^{\beta}<\infty $ it holds that, for all $%
x>0$,
\begin{equation}
\mathbb{P}\left( \zeta ^{2}\left( {\varepsilon }\right) \geq \delta
^{2}\mathrm{Tr}\left( \widetilde{{A}}\right) +2\delta ^{2}\sqrt{%
\mathrm{Tr}\left( \widetilde{{A}}\right) \rho \left( \widetilde{%
{A}}\right) x}+\delta ^{2}\rho \left( \widetilde{{A}}\right)
x\right) \leq C_2\left( \beta\right) \frac{\mathbb{E}\left\Vert \varepsilon
_{1}\right\Vert ^{\beta}\mathrm{Tr}\left( \widetilde{{A}}\right) }{\delta
^{\beta}\rho \left( \widetilde{{A}}\right) x^{\beta/2}},
\label{BoundCor5.1Baraud}
\end{equation}%
where the constant $C_2\left( \beta\right) $ depends only on $\beta.$
\end{prop}

\section{Appendix} \label{appen}
\subsection{Proof of Proposition \ref{casgen} }
This proof follows the guidelines of the proof of Theorem 6.1 in \cite{BAR}.
The following lemma will be helpful for the proof of this proposition
\begin{lem}
\label{lem}
Choose $\eta=\eta \left(\theta \right)>0$ and  $\alpha =\alpha \left( \theta \right) \in \left] 0 ; 1 \right[$ such that $\left(1+\theta \right)\left(1-\alpha\right) \geqslant \left(1+2\eta \right)$.
Set 
$H_m\left(f \right)= \left\lbrace \left\Vert f- \widetilde{f} \right\Vert_n ^ 2 - \tilde{\kappa} \left( \theta \right)  \left[ \left\Vert f- f_m\right\Vert^2_n + \frac{D_m}{n}\widehat{\delta}_m ^2 \right] \right\rbrace_+$ where $\tilde{\kappa} \left( \theta \right) = \left(2 + \frac{4}{\eta}\right)\left(1+\theta \right) $.  
Then, for ${m_0}$ minimizing  $m \mapsto \left\Vert f- f_m\right\Vert^2_n + \frac{D_m}{n}{\delta}_m ^2$  in $ m \in \mathcal{M}$
\begin{equation}
\label{eqlem}
\E \left[ H_{m_0}\left(f \right)^q \mathbf{1}_\Omega\right] \leqslant  \Delta^q_ \kappa \delta^{2q}_{sup} \frac{1}{n^q}.
\end{equation}
where $\Delta_ \kappa$ was defined in Proposition \ref{casgen}.
\end{lem}

\begin{proof}{\textbf{Lemma \ref{lem}}}

First, remark that on the set $\Omega$, for all $m\in \mathcal{M}$  
$$\widehat{pen}\left(m \right) \geqslant (1-\alpha )\left(1+ \theta \right)\frac{{\delta} _{m}^{2}D_{m}}{n} \geqslant (1+2\eta)\frac{{\delta} _{m}^{2}D_{m}}{n}.$$
Set $pen(m)=(1+2\eta)\frac{{\delta} _{m}^{2}D_{m}}{n}$, which corresponds to the penalty of Proposition \ref{ExtTh3.1Baraud}.

The proof of this lemma is based on the proof of Proposition \ref{ExtTh3.1Baraud} in \cite{MR2684389}.
In fact, it is sufficient to prove 
that for each $x>0$ and $ \kappa\geq 2$%
\begin{equation}
\mathbb{P}\left( \mathcal{H}\left( {f}\right)\mathbf{1}_\Omega \geq \left( 1+\frac{2}{%
\eta }\right) \frac{x}{n}\delta _{m}^{2}\right) \leq c\left(  \kappa,\eta \right)
\mathbb{E}\left\Vert {\varepsilon }_{1}\right\Vert
^{ \kappa}\sum\limits_{m\in \mathcal{M}}\frac{1}{\delta _{m}^{\kappa}}\frac{D_{m}\vee
1}{\left({%
\eta}D_{m}+x\right) ^{ \kappa/2}},  \label{ineqPH}
\end{equation}%
where we have set
\begin{equation*}
\mathcal{H}\left( {f}\right) =\left[ \left\Vert {f}-\widetilde{%
{f}}\right\Vert _{n}^{2}-\left( 2+\frac{4}{\eta }\right)\left\lbrace\left\Vert f- f_{m_0}
\right\Vert^2_n +\widehat{pen}\left( m_0\right) \right\rbrace \right] _{+}.
\end{equation*}%
Indeed, for
each $m\in \mathcal{M}$,
\begin{align*}
\left\Vert f- f_{m_0} \right \Vert ^2_n +\widehat{pen}\left( m_0\right) & =
\left\Vert f- f_{m_0} \right \Vert ^2_n +\left(1+ \theta \right)
\frac{\widehat{\delta }_{{m_0}}^{2}}{n}D_{{m_0}} \\
& \leq \left( 1+\theta \right) \left( \left\Vert f- f_{m_0} \right \Vert ^2_n +\frac{\widehat{\delta} _{{m_0}}^{2}}{n}D_{{m_0}}\right)
\end{align*}%
then we get that for all $q>0$,%
\begin{equation}
\mathcal{H}^{q}\left( {f}\right)\mathbf{1}_\Omega \geq  H_{m_0}^q\left(f \right)\mathbf{1}_\Omega \label{ineqHq}
\end{equation}%

Using the equality $$\E \left[ \mathcal{H}^{q}\left( {f}\right)\mathbf{1}_\Omega \right] = \int_0^\infty q u^{q-1} \mathbb{P} \left( \mathcal{H}^{q}\left( {f}\right)\mathbf{1}_\Omega > u \right) du$$ and following the proof of Propositon \ref{ExtTh3.1Baraud} in \cite{MR2684389} we obtain the upper bound \eqref{eqlem} of Lemma \ref{lem}.

Now we turn to the proof of \eqref{ineqPH}.
For any ${g}\in $ $\mathbb{R}^{np^2}$ we define the empirical quadratic loss
function by
\begin{equation*}
\gamma _{n}\left( {g}\right) =\left\Vert {y-g}\right\Vert
_{n}^{2}.
\end{equation*}%
Using the definition of $\gamma _{n}$ we have that for all ${g}\in $ $%
\mathbb{R}^{np^2}$,%

\begin{equation*}
\left\Vert {f}-{g}\right\Vert _{n}^{2}=\gamma _{n}\left(
{g}\right) +2\left\langle {g}-{y},{\ {%
\varepsilon }}\right\rangle _{n}+\left\Vert {\ {\varepsilon }}%
\right\Vert _{n}^{2}
\end{equation*}%
and therefore%
\begin{equation}
\left\Vert {f}-\widetilde{f}\right\Vert _{n}^{2}-\left\Vert
{f}-{P}_{m_0}{f}\right\Vert _{n}^{2}=\gamma _{n}\left(
\widetilde{f}\right) -\gamma _{n}\left( {P}_{m_0}{f}%
\right) +2\left\langle \widetilde{f}-{P}_{m_0}{f},%
{\ {\varepsilon }}\right\rangle _{n}.  \label{eq7}
\end{equation}

Using the definition of $\widetilde{f}$, we know that%
\begin{equation*}
\gamma _{n}\left( \widetilde{{f}}\right) +\widehat{pen}\left( \widetilde{m}%
\right) \leq \gamma _{n}\left( {g}\right) +\widehat{pen}\left( m_0\right)
\end{equation*}%
 for all ${g}\in \mathcal{S}_{m_0}$.
Then%
\begin{equation}
\gamma _{n}\left( \widetilde{{f}}\right) -\gamma _{n}\left( {P}%
_{m_0}{f}\right) \leq \widehat{pen}\left( m_0\right) -\widehat{pen}\left( \widetilde{m}\right) .
\label{eq7.1}
\end{equation}%
So we get from (\ref{eq7}) and (\ref{eq7.1}) that%
\begin{align}
\left\Vert {f}-\widetilde{f}\right\Vert _{n}^{2}\leq
& \left\Vert {f}-{P}_{m_0}{f}\right\Vert _{n}^{2}+  \widehat{pen}\left(
m_0\right) -\widehat{pen}\left( \widetilde{m}\right) \nonumber \\
 & +2\left\langle {f}-{P}%
_{m_0}{f},{\ {\varepsilon }}\right\rangle
_{n}+2\left\langle {P}_{\widetilde{m}}{f}-{f},{\
{\varepsilon }}\right\rangle _{n}+2\left\langle \widetilde{f}%
-{P}_{\widetilde{m}}{f},{\ {\varepsilon }}%
\right\rangle _{n}.  \label{eq8}
\end{align}%
In the following we set for each $m^{\prime }\in \mathcal{M}$,
\begin{align*}
\mathcal{B}_{m^{\prime }}& =\left\{ {g}\in \mathcal{S}_{m^{\prime
}}:\left\Vert {g}\right\Vert _{n}\leq 1\right\} , \\
G_{m^{\prime }}& =\underset{t\in \mathcal{B}_{m^{\prime }}}{\sup }%
\left\langle {g},{\ {\varepsilon }}\right\rangle
_{n}=\left\Vert {P}_{m^{\prime }}{\ {\varepsilon }}%
\right\Vert _{n}, \\
& {u}_{m^{\prime }}=%
\begin{cases}
\frac{{P}_{m^{\prime }}{f}-{f}}{\left\Vert {P}%
_{m^{\prime }}{f}-{f}\right\Vert _{n}} & \text{ if }\left\Vert
{P}_{m^{\prime }}{f}-{f}\right\Vert _{n}\neq 0 \\
0 & \text{ otherwise.}%
\end{cases}%
\end{align*}%
Since $\widetilde{f}=$ ${P}_{\widetilde{m}}$ ${f}+$ $%
{P}_{\widetilde{m}}$\textbf{\ ${\varepsilon }$}, (\ref{eq8})
gives%
\begin{equation*}
\left\Vert {f}-\widetilde{f}\right\Vert _{n}^{2}\leq
\left\Vert {f}-{P}_{m_0}{f}\right\Vert _{n}^{2}+\widehat{pen}\left(
m_0\right) -\widehat{pen}\left( \widetilde{m}\right)
\end{equation*}%
\begin{equation}
+2\left\Vert {f}-{P}_{m_0}{f}\right\Vert _{n}\left\vert
\left\langle {u}_{m_0},{\ {\varepsilon }}\right\rangle
_{n}\right\vert +2\left\Vert {f}-{P}_{\widetilde{m}}{f}%
\right\Vert _{n}\left\vert \left\langle {u}_{\widetilde{m}},{\
{\varepsilon }}\right\rangle _{n}\right\vert +2G_{\widetilde{m}}^{2}.
\label{eq9}
\end{equation}%
Using repeatedly the following elementary inequality that holds for all
positive numbers $\nu ,x,z$%
\begin{equation}
2xz\leq \nu x^{2}+\frac{1}{\nu }z^{2}  \label{eq101}
\end{equation}%
we get for any $m^{\prime }\in \mathcal{M}$%
\begin{equation}
2\left\Vert {f}-{P}_{m^{\prime }}{f}\right\Vert
_{n}\left\vert \left\langle {u}_{m^{\prime }},{\ {%
\varepsilon }}\right\rangle _{n}\right\vert \leq \nu \left\Vert {f}%
-{P}_{m^{\prime }}{f}\right\Vert _{n}^{2}+\frac{1}{\nu }%
\left\vert \left\langle {u}_{m^{\prime }},{\ {%
\varepsilon }}\right\rangle _{n}\right\vert ^{2}.  \label{eq11}
\end{equation}%
By Pythagora's Theorem we have
\begin{align}
\left\Vert {f}-\widetilde{f}\right\Vert _{n}^{2}&
=\left\Vert {f}-{P}_{\widetilde{m}}{f}\right\Vert
_{n}^{2}+\left\Vert {P}_{\widetilde{m}}{f}-\widetilde{f}%
\right\Vert _{n}^{2}  \notag \\
& =\left\Vert {f}-{P}_{\widetilde{m}}{f}\right\Vert
_{n}^{2}+G_{\widetilde{m}}^{2}.  \label{eq12}
\end{align}%
We derive from (\ref{eq9}) and (\ref{eq11}) that for any $\nu >0$ %
\begin{equation*}
\left\Vert {f}-\widetilde{f}\right\Vert _{n}^{2}\leq
\left\Vert {f}-{P}_{m_0}{f}\right\Vert _{n}^{2}+\nu
\left\Vert {f}-{P}_{m_0}{f}\right\Vert _{n}^{2}+\frac{1}{%
\nu }\left\langle {u}_{m_0},{\ {\varepsilon }}%
\right\rangle _{n}^{2}
\end{equation*}%
\begin{equation*}
+\nu \left\Vert {f}-{P}_{\widetilde{m}}{f}\right\Vert
_{n}^{2}+\frac{1}{\nu }\left\langle {u}_{\widetilde{m}},{\
{\varepsilon }}\right\rangle _{n}^{2}+2G_{\widetilde{m}}^{2}+\widehat{pen}\left(
m_0\right) -\widehat{pen}\left( \widetilde{m}\right) .
\end{equation*}%
Now taking into account that by equation (\ref{eq12}) $\left\Vert {f}-%
{P}_{\widetilde{m}}{f}\right\Vert _{n}^{2}=\left\Vert {f}-%
\widetilde{f}\right\Vert _{n}^{2}-G_{\widetilde{m}}^{2}$ the above
inequality is equivalent to %
\begin{equation*}
\left( 1-\nu \right) \left\Vert {f}-\widetilde{f}%
\right\Vert _{n}^{2}\leq \left( 1+\nu \right) \left\Vert {f}-%
{P}_{m_0}{f}\right\Vert _{n}^{2}+\frac{1}{\nu }\left\langle
{u}_{m_0},{\varepsilon }\right\rangle _{n}^{2}
\end{equation*}%
\begin{equation}
+\frac{1}{\nu }\left\langle {u}_{\widetilde{m}},{\varepsilon }%
\right\rangle _{n}^{2}+\left( 2-\nu \right) G_{\widetilde{m}}^{2}+\widehat{pen}\left(
m_0\right) -\widehat{pen}\left( \widetilde{m}\right) .  \label{eq13}
\end{equation}%
We choose $\nu =\frac{2}{2+\eta }\in \left] 0,1\right[ $, but for sake of
simplicity we keep using the notation $\nu $. Let $\widetilde{p}_{1}$ and
$\widetilde{p}_{2}$ be two functions depending on $\nu $ mapping $\mathcal{M%
}$ into $\mathbb{R}_{+}$. They will be specified as in \cite{MR2684389} to satisfy%
\begin{equation}
pen\left( m^{\prime }\right) \geq \left( 2-\nu \right) \widetilde{p}%
_{1}\left( m^{\prime }\right) +\frac{1}{\nu }\widetilde{p}_{2}\left(
m^{\prime }\right) \text{ }\forall (m^{\prime })\in \mathcal{M}_{.}
\label{eq14}
\end{equation}%
Remember that on $\Omega$, $\widehat{pen}(m) \geqslant pen(m)$ $\forall m \in \mathcal{M}$. Since $\frac{1}{\nu }\widetilde{p}_{2}\left( m^{\prime }\right) \leq
pen\left( m^{\prime }\right) $ and $1+\nu \leq 2$, we get from (\ref{eq13}%
) and (\ref{eq14}) that on the set  $\Omega$ %
\begin{align}
\left( 1-\nu \right) \left\Vert {f}-\widetilde{f}%
\right\Vert _{n}^{2}& \leq \left( 1+\nu \right) \left\Vert {f}-%
{P}_{m_0}{f}\right\Vert _{n}^{2}+\widehat{pen}\left( m_0\right) +\frac{1}{%
\nu }\widetilde{p}_{2}\left( m_0\right) +\left( 2-\nu \right) \left( G_{%
\widetilde{m}}^{2}-\widetilde{p}_{1}\left( \widetilde{m}\right) \right)   \notag
\\
& +\frac{1}{\nu }\left( \left\langle {u}_{\widetilde{m}},{%
\varepsilon }\right\rangle _{n}^{2}-\widetilde{p}_{2}\left( \widetilde{m}%
\right) \right) +\frac{1}{\nu }\left( \left\langle {u}_{m_0},{%
\varepsilon }\right\rangle _{n}^{2}-\widetilde{p}_{2}\left( m_0\right) \right)
\notag \\
& \leq 2\left( \left\Vert {f}-{P}_{m_0}{f}\right\Vert
_{n}^{2}+\widehat{pen}\left( m_0\right) \right) +\left( 2-\nu \right) \left( G_{%
\widetilde{m}}^{2}-\widetilde{p}_{1}\left( \widetilde{m}\right) \right)   \notag
\\
& +\frac{1}{\nu }\left( \left\langle {u}_{\widetilde{m}},{%
\varepsilon }\right\rangle _{n}^{2}-\widetilde{p}_{2}\left( \widetilde{m}%
\right) \right) +\frac{1}{\nu }\left( \left\langle {u}_{m_0},{%
\varepsilon }\right\rangle _{n}^{2}-\widetilde{p}_{2}\left( m_0\right) \right)
.  \label{eq15}
\end{align}%

As $\frac{2}{1-\nu }=2+\frac{4}{\eta }$ we obtain that%
\begin{align*}
\left( 1-\nu \right) \mathcal{H}\left( {f}\right)\mathbf{1}_\Omega & =\left\{
\left( 1-\nu \right) \left\Vert {f}-\widetilde{f}%
\right\Vert _{n}^{2}-\left( 1-\nu \right) \left( 2+\frac{4}{\eta }\right)
\left( \left\Vert {f}-%
{P}_{m_0}{f}\right\Vert _{n}^{2}+\widehat{pen}\left( m_0\right) \right) \right\} _{+}\mathbf{1}_\Omega \\
& =\left\{ \left( 1-\nu \right) \left\Vert {f}-\widetilde{{f%
}}\right\Vert _{n}^{2}-2\left( \left\Vert {f}-{P}_{m_0}{f}%
\right\Vert _{n}^{2}+\widehat{pen}\left( m_0\right) \right) \right\} _{+} \mathbf{1}_\Omega\\
& \leq \left\{ \left( 2-\nu \right) \left( G_{\widetilde{m}}^{2}-\widetilde{%
p}_{1}\left( \widetilde{m}\right) \right) +\frac{1}{\nu }\left(
\left\langle {u}_{\widetilde{m}},{\varepsilon }\right\rangle
_{n}^{2}-\widetilde{p}_{2}\left( \widetilde{m}\right) \right) +\frac{1}{\nu
}\left( \left\langle {u}_{m},{\varepsilon }\right\rangle
_{n}^{2}-\widetilde{p}_{2}\left( m_0\right) \right) \right\} _{+}
\end{align*}%

For any $x>0,$%
\begin{align}
\mathbb{P}\left( \left( 1-\nu \right) \mathcal{H}\left( {f}\right)\mathbf{1}_\Omega
\geq \frac{x\delta _{m}^{2}}{n}\right) & \leq \mathbb{P}\left( \exists
m^{\prime }\in \mathcal{M}:\left( 2-\nu \right) \left( G_{m^{\prime
}}^{2}-\widetilde{p}_{1}\left( m^{\prime }\right) \right) \geq \frac{x\delta
_{m^{\prime }}^{2}}{3n}\right)   \notag \\
& +\mathbb{P}\left( \exists m^{\prime }\in \mathcal{M}:\frac{1}{\nu }%
\left( \left\langle {u}_{m^{\prime }},{\varepsilon }%
\right\rangle _{n}^{2}-\widetilde{p}_{2}\left( m^{\prime }\right) \right)
\geq \frac{x\delta _{m^{\prime }}^{2}}{3n}\right)   \notag \\
& \leq \sum\limits_{m^{\prime }\in \mathcal{M}}\mathbb{P}\left( \left(
2-\nu \right) \left( \left\Vert {P}_{m^{\prime }}{%
\varepsilon }\right\Vert _{n}^{2}-\widetilde{p}_{1}\left( m^{\prime }\right)
\right) \geq \frac{x\delta _{m^{\prime }}^{2}}{3n}\right)   \notag \\
& +\sum\limits_{m^{\prime }\in \mathcal{M}}\mathbb{P}\left( \frac{1}{\nu }%
\left( \left\langle {u}_{m^{\prime }},{\varepsilon }%
\right\rangle _{n}^{2}-\widetilde{p}_{2}\left( m^{\prime }\right) \right)
\geq \frac{x\delta _{m^{\prime }}^{2}}{3n}\right)   \notag \\
& :=\sum\limits_{m^{\prime }\in \mathcal{M}}P_{1,m^{\prime }}\left( x\right)
+\sum\limits_{m^{\prime }\in \mathcal{M}}P_{2,m^{\prime }}\left( x\right) .
\label{eq16}
\end{align}%

From now on, the proof of Lemma \ref{lem} is exactly the same as the end of the proof of Proposition \ref{ExtTh3.1Baraud} in \cite{MR2684389} with $L_m=\nu$.
\end{proof}

\begin{proof}{\textbf{Proposition \ref{casgen}}}

We first provide an upper bound for $\E \left[ \left\Vert f- \widetilde{f} \right \Vert^{2q}_n \mathbf{1}_{\Omega} \right]$, where the set $\Omega$ depends on $\alpha$ chose as in Lemma \ref{lem}.

As $q \leqslant 1$, we have $ \left( a + b \right) ^q \leqslant a^q + b^q$. Together with Lemma \ref{lem} we deduce that
$$\E \left[ \left\Vert f- \widetilde{f} \right \Vert^{2q}_n \mathbf{1}_{\Omega} \right]
\leqslant
\Delta^q_ \kappa \delta^{2q}_{sup} \frac{1}{n^q} + \E \left[ \tilde{\kappa} \left( \theta \right)^q  \left[ \left\Vert f- f_{m_0}\right\Vert^2_n + \frac{D_{m_0}}{n}\widehat{\delta}_{m_0} ^2 \right]^q \right].$$
Using the convexity of $x \mapsto x^\frac{1}{q}$ together with the Jensen inequality, we obtain
$$\left(\E \left[ \left\Vert f- \widetilde{f} \right \Vert^{2q}_n \mathbf{1}_{\Omega} \right]\right)^\frac{1}{q}
\leqslant
2^{1/q -1}  \Delta_ \kappa \delta^2_{sup} \frac{1}{n} + 2^{1/q -1} \E \left[ \tilde{\kappa} \left( \theta \right)  \left[ \left\Vert f- f_{m_0}\right\Vert^2_n + \frac{D_{m_0}}{n}\widehat{\delta}_{m_0} ^2 \right]\right],$$
and by using the assumption A1 we have that
\begin{equation}
\label{p1}
\left(\E \left[ \left\Vert f- \widetilde{f} \right \Vert^{2q}_n \mathbf{1}_{\Omega} \right]\right)^\frac{1}{q}
\leqslant2^{1/q -1} \Delta_ \kappa \delta^2_{sup} \frac{1}{n} + 2^{1/q -1} \tilde{\kappa} \left( \theta \right)  \left[ \left\Vert f- f_{m_0}\right\Vert^2_n + \frac{D_{m_0}}{n}{\delta}_{m_0} ^2 \right].
\end{equation}

Now we need to find an upper bound for the quantity
$\E \left[ \left\Vert f- \widetilde{f} \right \Vert^{2q}_n \mathbf{1}_{\Omega^c} \right].$

First, remark that 
$$\left\Vert f- \widetilde{f} \right \Vert^{2}_n = \left\Vert f- P_{\widetilde{m}}y \right \Vert^{2}_n=\left\Vert f- P_{\widetilde{m}}f \right \Vert^{2}_n +\left\Vert P_{\widetilde{m}}\left(f-y\right) \right \Vert^{2}_n$$
$$\leqslant\left\Vert f- P_{\widetilde{m}}f \right \Vert^{2}_n +\left\Vert \varepsilon \right \Vert^{2}_n= \left\Vert f\right \Vert^{2}_n-\left\Vert P_{\widetilde{m}}f \right \Vert^{2}_n +\left\Vert \varepsilon \right \Vert^{2}_n$$
And thus
$$\left\Vert f- \widetilde{f} \right \Vert^{2}_n \leqslant \left\Vert f\right \Vert^{2}_n +\left\Vert \varepsilon \right \Vert^{2}_n .$$
So we have
$$\E \left[ \left\Vert f- \widetilde{f} \right \Vert^{2q}_n \mathbf{1}_{\Omega^c} \right] \leqslant \left\Vert f\right \Vert^{2q}_n \mathbb{P}\left( \Omega^c \right) +\E \left[ \left\Vert \varepsilon \right \Vert^{2q}_n \mathbf{1}_{\Omega^c} \right].$$
Using H\"older's inequality with $\frac{ \kappa}{2q} > 1$ we obtain   
$$\E \left[ \left\Vert \varepsilon \right \Vert^{2q}_n \mathbf{1}_{\Omega^c} \right] \leqslant \E \left[ \left\Vert \varepsilon \right \Vert^ \kappa_n\right]^\frac{2q}{ \kappa} \mathbb{P}\left(\Omega^c \right)^{\left(1-\frac{2q}{ \kappa}\right)}. $$

But  $$\E \left[\left\Vert \varepsilon \right\Vert^ \kappa_n \right] = \frac{1}{n^\frac{ \kappa}{2}} \E\left[ \left( \sum_{i=1}^n \left\Vert \varepsilon_i \right\Vert^2 \right)^\frac{ \kappa}{2} \right],$$
and as $ \kappa \geqslant 2$, we can use Minkowsky's inequality to obtain
$$\E \left[\left\Vert \varepsilon \right\Vert^ \kappa_n \right] \leqslant\frac{1}{n^\frac{ \kappa}{2}} \left( \sum_{i=1}^n \left( E \left[ \left\Vert \varepsilon_i \right\Vert^ \kappa \right]  \right)^\frac{2}{ \kappa} \right)^\frac{ \kappa}{2} = \frac{1}{n^\frac{ \kappa}{2}} \left( n\left( E \left[ \left\Vert \varepsilon_1 \right\Vert^ \kappa \right]  \right)^\frac{2}{ \kappa} \right)^\frac{ \kappa}{2},$$
that is
$$\E \left[\left\Vert \varepsilon \right\Vert^ \kappa_n \right] \leqslant \E \left[ \left\Vert \varepsilon_1 \right\Vert^ \kappa\right].$$
So we have 
$$\E \left[ \left\Vert f- \widetilde{f} \right \Vert^{2q}_n \mathbf{1}_{\Omega^c} \right] \leqslant \left[ \E \left[ \left\Vert \varepsilon_1 \right\Vert^ \kappa\right]^\frac{2q}{ \kappa}
 + \left\Vert f\right \Vert^{2q}_n \right] \mathbb{P}\left(\Omega^c \right)^{\left(1-\frac{2q}{ \kappa}\right)},$$
 and with assumption A2
 $$\E \left[ \left\Vert f- \widetilde{f} \right \Vert^{2q}_n \mathbf{1}_{\Omega^c} \right] \leqslant \left[ \E \left[ \left\Vert \varepsilon_1 \right\Vert^ \kappa\right]^\frac{2q}{ \kappa}
 + \left\Vert f\right \Vert^{2q}_n \right] \left(  \tilde{C}(\alpha)\frac{1}{n^{\gamma}} \right)^{\left(1-\frac{2q}{ \kappa}\right)}.$$
 As $\gamma \geqslant\frac{q}{1- 2q/ \kappa}$, we deduce that
 \begin{equation}
 \label{p2}
 \left(\E \left[ \left\Vert f- \widetilde{f} \right \Vert^{2q}_n \mathbf{1}_{\Omega^c} \right]\right)^ \frac{1}{q} \leqslant 2^{\frac{1}{q}- 1} \left[ \E \left[ \left\Vert \varepsilon_1 \right\Vert^ \kappa\right]^\frac{2}{ \kappa}
 + \left\Vert f\right \Vert^{2}_n \right]   \tilde{C}(\alpha)^{\frac{1-\frac{2q}{ \kappa}}{q}}\frac{1}{n}
 \end{equation}
To conclude, we use again the convexity of $x \mapsto x^\frac{1}{q}$ and the inequality \eqref{p1} to get 
\begin{align*}
 \left(\E \left[ \left\Vert f- \widetilde{f} \right \Vert^{2q}_n \right]\right)^ \frac{1}{q} \leqslant 4^{\frac{1}{q}- 1} \left[ \E \left[ \left\Vert \varepsilon_1 \right\Vert^ \kappa\right]^\frac{2}{ \kappa} 
 + \left\Vert f\right \Vert^{2}_n \right]   & \tilde{C}(\alpha)^{\frac{1-\frac{2q}{ \kappa}}{q}}\frac{1}{n}\\ + 4^{1/q -1} \Delta_ \kappa \delta^2_{sup} \frac{1}{n} \\ +4^{1/q -1} \tilde{\kappa} \left( \theta \right)  \left[ \left\Vert f- f_{m_0}\right\Vert^2_n + \frac{D_{m_0}}{n}{\delta}_{m_0} ^2 \right] 
 \end{align*}
\end{proof}

\subsection{Proof of Theorem \ref{theochili} }
Recall that $\beta > \max \left( 2\left( 1+2q\right), 3+2q\right)$ and  $ \kappa \in  \left] 2\left( 1+2q\right) ; \min\left( \beta, 2 \beta -4\right) \right[$.
 In order to use Proposition \ref{casgen} , we need to prove the following inequalities :
\begin{enumerate}
\item[ A1.] $\E \left[ \widehat{\delta}_m^2 \right] \leqslant \delta_m^2$
\item[ A2.]$\mathbb{P} \left( \Omega^ c \right) \leqslant \tilde{C}\left( \alpha \right)\frac{1}{n^\gamma}$ for $\gamma \geqslant \frac{q}{1- 2q/\kappa}$. %
 \end{enumerate}

First we prove A1.

Remember that $\widehat{\delta} _{m}^{2} =\frac{\mathrm{Tr}\left( \left( {\Pi }_{m}\otimes 
{\Pi }_{m}\right) {\widehat{\Phi} }\right) }{D_{m}}$. By using the linearity of the trace and the equality $\E\left[ \widehat{\Phi} \right] = \frac{n-1}{n} \Phi$, we obtain that $\E\left[ \widehat{\delta} _{m}^{2}\right] = \frac{n-1}{n} {\delta} _{m}^{2}$ which proves the result.
\\

For the second, write $\Omega^c = \cup_{m \in \mathcal{M}} \left\lbrace\widehat{\delta}_m^2 \leqslant \left(1 - \alpha \right)\delta_m^2 \right\rbrace$.
We bound up the quantity $\mathbb{P}\left(\widehat{\delta}_m^2 \leqslant \left(1 - \alpha \right)\delta_m^2 \right)$ in the following Proposition.

\begin{prop}
\label{lemchili}
For all $m \in \mathcal{M}$ , $\alpha \in ]0;1[$ and $n\geqslant n(\kappa,\beta,\alpha,C_{inf},\Sigma)$  we have for some constants $C_1\left( \beta \right)$, $C_2\left( \beta \right)$ :
$$\mathbb{P}\left(\widehat{\delta}_m^2 \leqslant \left(1 - \alpha \right)\delta_m^2 \right)\leqslant \frac{1}{n^{\gamma}} \left(
C_2(\beta) 2 ^ {\beta  + 1} +  C_1 \left( \beta \right) \frac{1 }{\alpha^\frac{\beta}{2} }\right)\E \left[ \left\Vert x x^\top  \right\Vert^\beta \right] \delta_m^{-\beta} D_m^{-\frac{\beta}{2}},$$
for $\gamma \geqslant \frac{q}{1- 2q/\kappa}$.
\end{prop}

This Proposition concludes the proof of A2 with $$\tilde{C}\left(\alpha \right) = \left(
C_2(\beta) 2 ^ {\beta  + 1} +  C_1 \left( \beta \right) \frac{1 }{\alpha^\frac{\beta}{2} }\right)\E \left[ \left\Vert x x^\top  \right\Vert^\beta \right] \sum_{m \in \mathcal{M}} \delta_m^{-\beta} D_m^{-\frac{\beta}{2}}.$$
\begin{proof}{\textbf{Proposition~\ref{lemchili}}}
%
%

We start by dividing $\mathcal{P}_m=\mathbb{P} \left( \widehat{\delta}_m^2 \leqslant \left(1 - \alpha \right)\delta_m^2 \right)$ into two parts with one of them involving a sum of independent variables with expectation equal to 0.
$$\mathcal{P}_m
= \mathbb{P} \left( \mathrm{Tr}\left( \left( {\Pi }_{m}\otimes 
{\Pi }_{m}\right) {\widehat{\Phi} }\right) \leqslant \left(1 - \alpha \right)\mathrm{Tr}\left( \left( {\Pi }_{m}\otimes 
{\Pi }_{m}\right) {{\Phi} }\right) \right)$$

$$ \mathcal{P}_m= \mathbb{P} \left( \mathrm{Tr}\left( \left( {\Pi }_{m}\otimes 
{\Pi }_{m}\right) \left(\widehat{\Phi} -\left(\Phi+ \mu \mu^\top \right) + \mu \mu ^\top\right)\right) \leqslant  - \alpha \mathrm{Tr}\left( \left( {\Pi }_{m}\otimes 
{\Pi }_{m}\right) {{\Phi} } \right)\right)$$

$$\mathcal{P}_m\leqslant\mathbb{P} \left( \left\vert \mathrm{Tr}\left( \left( {\Pi }_{m}\otimes 
{\Pi }_{m}\right) \left(\frac{1}{n}\sum_{i=1}^n \left(y_i y_i^T -\Phi - \mu \mu^ \top \right) + \mu \mu ^\top - S_{vec} S_{vec}^\top\right) \right) \right\vert \geqslant   \alpha \mathrm{Tr}\left( \left( {\Pi }_{m}\otimes 
{\Pi }_{m}\right){{\Phi} }\right) \right)$$

$$\mathcal{P}_m\leqslant\mathbb{P} \left( \left\vert \mathrm{Tr}\left( \left( {\Pi }_{m}\otimes 
{\Pi }_{m}\right) \left(\frac{1}{n}\sum_{i=1}^n \left(y_i y_i^T -\Phi - \mu \mu^ \top \right) \right) \right) \right\vert \geqslant   \frac{\alpha}{2} \mathrm{Tr}\left( \left( {\Pi }_{m}\otimes 
{\Pi }_{m}\right){{\Phi} }\right) \right)$$
$$+\mathbb{P} \left( \left\vert \mathrm{Tr}\left( \left( {\Pi }_{m}\otimes 
{\Pi }_{m}\right) \left( \mu \mu ^\top - S_{vec} S_{vec}^\top\right) \right) \right\vert \geqslant  \frac{ \alpha}{2} \mathrm{Tr}\left( \left( {\Pi }_{m}\otimes 
{\Pi }_{m}\right){{\Phi} }\right) \right)$$

Set 
$$Q1=\mathbb{P} \left( \left\vert \mathrm{Tr}\left( \left( {\Pi }_{m}\otimes 
{\Pi }_{m}\right) \left(\frac{1}{n}\sum_{i=1}^n \left(y_i y_i^T -\Phi - \mu \mu^ \top \right) \right) \right) \right\vert \geqslant   \frac{\alpha}{2} \mathrm{Tr}\left( \left( {\Pi }_{m}\otimes 
{\Pi }_{m}\right){{\Phi} }\right) \right)$$ and
$$Q2=\mathbb{P} \left( \left\vert \mathrm{Tr}\left( \left( {\Pi }_{m}\otimes 
{\Pi }_{m}\right) \left( \mu \mu ^\top - S_{vec} S_{vec}^\top\right) \right) \right\vert \geqslant  \frac{ \alpha}{2} \mathrm{Tr}\left( \left( {\Pi }_{m}\otimes 
{\Pi }_{m}\right){{\Phi} }\right) \right) $$

\textbf{Study of Q1}

First we use Markov's inequality to obtain 
$$A1 \leqslant \frac{2^\frac{\beta }{2}\E \left[\left\vert \mathrm{Tr}\left( \left( {\Pi }_{m}\otimes 
{\Pi }_{m}\right) \left(\frac{1}{n}\sum_{i=1}^n \left(y_i y_i^T -\Phi - \mu \mu^ \top \right)\right) \right) \right\vert^\frac{\beta}{2} \right]}{\left(\alpha \mathrm{Tr}\left( \left( {\Pi }_{m}\otimes 
{\Pi }_{m}\right){\Phi}\right)\right) ^\frac{\beta}{2}}.$$

We must consider the two following cases :
\begin{itemize}

\item If $\frac{\beta}{2} \geqslant 2$, Rosenthal's inequality gives 
\begin{align*}\E &\left[\left\vert \frac{1}{n}\sum_{i=1}^n\mathrm{Tr}\left( \left( {\Pi }_{m}\otimes 
{\Pi }_{m}\right) \left( \left(y_i y_i^T -\Phi - \mu \mu^ \top \right)\right) \right) \right\vert^\frac{\beta}{2} \right]\\ \leqslant &
C \left( \frac{\beta}{2} \right) \frac{1}{n^{\frac{\beta}{2} - 1}} \E\left[ \left\vert \mathrm{Tr}\left( \left( {\Pi }_{m}\otimes 
{\Pi }_{m}\right) \left({y_1 y_1^\top -\Phi - \mu \mu^ \top }\right) \right) \right\vert ^\frac{\beta}{2} \right]   \\ 
+  C \left( \frac{\beta}{2} \right)&\left( \frac{1}{n}\E\left[ \left\vert \mathrm{Tr}\left( \left( {\Pi }_{m}\otimes 
{\Pi }_{m}\right) \left({y_1y_1^\top - \Phi - \mu \mu^ \top }\right) \right) \right\vert ^2 \right]\right)^\frac{\beta}{4}.
\end{align*}

As $\frac{\beta}{2}\geqslant 2$, $\frac{1}{n^{\frac{\beta}{2}- 1}} \leqslant \frac{1}{n^ \frac{\beta}{4}}$ and we can use Jensen's inequality on the second term to obtain
$$\E \left[\left\vert \frac{1}{n}\sum_{i=1}^n\mathrm{Tr}\left( \left( {\Pi }_{m}\otimes 
{\Pi }_{m}\right) \left( \left(y_i y_i^T -\Phi - \mu \mu^ \top \right)\right) \right) \right\vert^\frac{\beta}{2} \right]$$ 
$$\leqslant C\left(\frac{\beta}{2} \right)\E\left[ \left\vert \mathrm{Tr}\left( \left( {\Pi }_{m}\otimes 
{\Pi }_{m}\right) \left(y_1y_1^\top - \Phi - \mu \mu^ \top \right) \right) \right\vert ^\frac{\beta}{2} \right] \frac{2}{n^\frac{\beta}{4}}.$$

\item If $1 \leqslant \frac{\beta}{2} \leqslant 2$, we use Lemma \ref{paraderose} of  subsection \ref{auxlem} to get 
\begin{align*}\E &\left[\left\vert \frac{1}{n}\sum_{i=1}^n\mathrm{Tr}\left( \left( {\Pi }_{m}\otimes 
{\Pi }_{m}\right) \left( \left(y_i y_i^T -\Phi - \mu \mu^ \top \right)\right) \right) \right\vert^\frac{\beta}{2} \right]\\ \leqslant & \frac{8}{n^{\frac{\beta}{2}-1 }} \E\left[ \left\vert \mathrm{Tr}\left( \left( {\Pi }_{m}\otimes 
{\Pi }_{m}\right) \left({y_1 y_1^\top -\Phi - \mu \mu^ \top }\right) \right) \right\vert ^\frac{\beta}{2} \right] 
\end{align*}
\end{itemize}
In both cases, we can use the fact that $x \mapsto x^\frac{\beta}{2}$ is a convex and increasing function to obtain
$$\E\left[ \left\vert \mathrm{Tr}\left( \left( {\Pi }_{m}\otimes 
{\Pi }_{m}\right) \left(y_1y_1^\top - \Phi - \mu \mu^ \top \right) \right) \right\vert ^\frac{\beta}{2} \right]$$
$$\leqslant 2^{\frac{\beta}{2}-1}\left[ \E\left[ \left\vert \mathrm{Tr}\left( \left( {\Pi }_{m}\otimes 
{\Pi }_{m}\right) \left(y_1 y_1^\top \right) \right) \right\vert^\frac{\beta}{2} \right] + \left\vert \mathrm{Tr}\left( \left( {\Pi }_{m}\otimes 
{\Pi }_{m}\right) \left( \Phi +  \mu \mu^ \top \right)\right)\right\vert ^\frac{\beta}{2} \right] .$$
And by using the Jensen's inequality on the second term we have that
\begin{align*}
\E & \left[ \left\vert \mathrm{Tr}\left( \left( {\Pi }_{m}\otimes 
{\Pi }_{m}\right) \left(y_1y_1^\top - \Phi - \mu \mu^ \top \right) \right) \right\vert ^\frac{\beta}{2} \right] \\ &\leqslant 2^\frac{\beta}{2} \E\left[ \left\vert \mathrm{Tr}\left( \left( {\Pi }_{m}\otimes 
{\Pi }_{m}\right) \left(y_1 y_1^\top  \right) \right) \right\vert^\frac{\beta}{2} \right].
\end{align*}

Now consider the following lemma.
\begin{lem}
\label{evident}
If $\Psi$ is  symmetric non-negative definite, then 
\begin{equation}
\mathrm{Tr}\left( \left( {\Pi }_{m}\otimes 
{\Pi }_{m}\right) \Psi \right) \in \left[ 0; \mathrm{Tr}\left(\Psi \right)\right]
 \end{equation}
\end{lem}

From this fact we get that
$$\left\vert \mathrm{Tr}\left( \left( {\Pi }_{m}\otimes 
{\Pi }_{m}\right) \left(y_1 y_1^\top  \right) \right) \right\vert^\frac{\beta}{2} \leqslant \mathrm{Tr}\left( y_1 y_1^\top\right) ^\frac{\beta}{2} = \left\Vert y_1 \right\Vert^ {\beta} = \left\Vert x x^ \top\right\Vert^ {\beta}. $$

In conclusion, we have
\begin{equation}
\label{eqA1}
Q1 \leqslant C_1 \left( \beta \right)\frac{\E\left[ \left\Vert x x^ \top\right\Vert^ {\beta}\right] }{\alpha ^\frac{\beta}{2} \delta_m^{\beta} D_m^\frac{\beta}{2}} \frac{1}{n^{\gamma}},
\end{equation}
with ${\gamma} =\min\left(\frac{\beta}{4}, \frac{\beta}{2} -1 \right)$ and $C_1\left(\beta\right)=2C\left( \frac{\beta}{2}\right)$ if $\beta\geqslant 4$ where $C\left( \frac{\beta}{2}\right)$ is the constant in Rosenthal's inequality and $C_1\left(\beta\right)=8$ if $2\leqslant \beta \leqslant 4$. Remark that $ \frac{\beta}{4} \geqslant \frac{\kappa}{4}$ and $\frac{\beta}{2} -1 \geqslant \frac{\kappa}{4}$, so ${\gamma} \geqslant  \frac{\kappa}{4}$. 
\\

\textbf{Study of Q2}

Recall that $$ Q2=\mathbb{P} \left( \left\vert \mathrm{Tr}\left( \left( {\Pi }_{m}\otimes 
{\Pi }_{m}\right) \left( \mu \mu ^\top - S_{vec} S_{vec}^\top\right) \right) \right\vert \geqslant  \frac{ \alpha}{2} \mathrm{Tr}\left( \left( {\Pi }_{m}\otimes 
{\Pi }_{m}\right){{\Phi} }\right) \right).$$
Set $$B2= \mathrm{Tr}\left( \left( {\Pi }_{m}\otimes 
{\Pi }_{m}\right) \left( \mu \mu ^\top - S_{vec} S_{vec}^\top\right) \right).$$
Using the properties of the trace, we can write $$B2 =
\mathrm{Tr}\left( \left( {\Pi }_{m}\otimes 
{\Pi }_{m}\right) \left( \mu \mu ^\top\right)\right) - \mathrm{Tr}\left( \left( {\Pi }_{m}\otimes 
{\Pi }_{m}\right) \left(S_{vec} S_{vec}^\top\right) \right)$$

$$=\mathrm{Tr}\left(\mu ^\top \left( {\Pi }_{m}\otimes 
{\Pi }_{m}\right)  \mu \right) - \mathrm{Tr}\left( S_{vec}^\top\left( {\Pi }_{m}\otimes 
{\Pi }_{m}\right) S_{vec}  \right).
$$

But  ${\Pi }_{m}\otimes 
{\Pi }_{m}$ is an orthogonal projection matrix, then 

$$B2=\mathrm{Tr}\left(\mu ^\top \left( {\Pi }_{m}\otimes 
{\Pi }_{m} \right)^\top\left( {\Pi }_{m}\otimes 
{\Pi }_{m}\right)  \mu \right) - \mathrm{Tr}\left( S_{vec}^\top\left( {\Pi }_{m}\otimes 
{\Pi }_{m} \right)^\top\left( {\Pi }_{m}\otimes 
{\Pi }_{m}\right) S_{vec}  \right)$$
$$B2=\left\Vert \left( {\Pi }_{m}\otimes 
{\Pi }_{m} \right)\mu \right\Vert^2  - \left\Vert \left( {\Pi }_{m}\otimes 
{\Pi }_{m} \right)S_{vec} \right\Vert^2$$
$$B2= \left(\left\Vert \left( {\Pi }_{m}\otimes 
{\Pi }_{m} \right)\mu \right\Vert  - \left\Vert \left( {\Pi }_{m}\otimes 
{\Pi }_{m} \right)S_{vec} \right\Vert\right)\left(\left\Vert \left( {\Pi }_{m}\otimes 
{\Pi }_{m} \right)\mu \right\Vert  + \left\Vert \left( {\Pi }_{m}\otimes 
{\Pi }_{m} \right)S_{vec} \right\Vert\right)$$

Hence 
$$\left\vert B2 \right\vert
\leqslant\left\Vert \left( {\Pi }_{m}\otimes 
{\Pi }_{m} \right)\left(\mu -S_{vec}\right) \right\Vert\left(\left\Vert \left( {\Pi }_{m}\otimes 
{\Pi }_{m} \right)\mu \right\Vert  + \left\Vert \left( {\Pi }_{m}\otimes 
{\Pi }_{m} \right)S_{vec} \right\Vert\right)$$
$$\left\vert B2 \right\vert
\leqslant\left\Vert \left( {\Pi }_{m}\otimes 
{\Pi }_{m} \right)\left(\mu -S_{vec}\right) \right\Vert^2+2\left\Vert \left( {\Pi }_{m}\otimes 
{\Pi }_{m} \right)\left(\mu -S_{vec} \right)\right\Vert\left\Vert \left( {\Pi }_{m}\otimes 
{\Pi }_{m} \right)\mu  \right\Vert$$
$$\left\vert B2 \right\vert
\leqslant\left\Vert \left( {\Pi }_{m}\otimes 
{\Pi }_{m} \right)\left(\mu -S_{vec}\right) \right\Vert^2 + 2\left\Vert \left( {\Pi }_{m}\otimes 
{\Pi }_{m} \right)\left(\mu -S_{vec} \right)\right\Vert\left\Vert \mu  \right\Vert$$
Finally 
\begin{align}
Q2 \leqslant &\mathbb{P} \left(\left\Vert \left( {\Pi }_{m}\otimes 
{\Pi }_{m} \right)\left(\mu -S_{vec}\right) \right\Vert^2 \geqslant  \frac{ \alpha}{4} \mathrm{Tr}\left( \left( {\Pi }_{m}\otimes 
{\Pi }_{m}\right){{\Phi} }\right) \right)\\ \nonumber
+&\mathbb{P} \left(\left\Vert \left( {\Pi }_{m}\otimes 
{\Pi }_{m} \right)\left(\mu -S_{vec} \right)\right\Vert \geqslant  \frac{ \alpha}{8\left\Vert \mu  \right\Vert} \mathrm{Tr}\left( \left( {\Pi }_{m}\otimes 
{\Pi }_{m}\right){{\Phi} }\right) \right)
\end{align}

Now we need to provide an upper bound for the quantities  
$$\mathbb{P} \left(\left\Vert \left( {\Pi }_{m}\otimes 
{\Pi }_{m} \right)\left(\mu -S_{vec}\right) \right\Vert^2 \geqslant t \right).$$
\\

For this we will use the deviation bound provided by Proposition \ref{proBaraud}  stated in subsection \ref{auxlem} .

%

Set $$G_n= \frac{1}{n}\left( \begin{array}{ccc}
Id_{p^2} &  \dots & Id_{p^2} \\
Id_{p^2} &  \dots & Id_{p^2}\\
\vdots & \ddots & \vdots \\
Id_{p^2} &  \dots & Id_{p^2}
\end{array} \right) \in \R^{p^2n\times p^2n}$$
Then 
$$G_n \left(y- f\right) = \mathbf{1_n}\otimes \left(S_{vec}- \mu \right)$$

Now, if $$H_m= Id_n \otimes \left(\Pi_m \otimes \Pi_m\right)=\left( \begin{array}{cccc}
\Pi_m \otimes \Pi_m  & 0 &  \dots & 0 \\
0 & \Pi_m \otimes \Pi_m & \dots & 0\\
\vdots & \vdots & \ddots & \vdots\\
0 & 0 &  \dots & \Pi_m \otimes \Pi_m 
\end{array} \right) \in \R^{p^2n\times p^2n},$$

we have $$H_m\left(\mathbf{1}_n\otimes \left(S_{vec}- \mu \right)\right)= \mathbf{1}_n\otimes\left( \left(\Pi_m \otimes \Pi_m\right) \left(S_{vec}- \mu \right)\right).$$

In conclusion, with $$A_m = H_m G_n = \frac{1}{n}\left( \begin{array}{ccc}
\Pi_m \otimes \Pi_m&  \dots & \Pi_m \otimes \Pi_m \\
\Pi_m \otimes \Pi_m &  \dots & \Pi_m \otimes \Pi_m\\
\vdots & \ddots & \vdots \\
\Pi_m \otimes \Pi_m &  \dots &\Pi_m \otimes \Pi_m
\end{array} \right) \in \R^{p^2n\times p^2n},$$
we have that 
$$A_m \left(y - f\right) = \mathbf{1}_n \otimes \left( \left(\Pi_m \otimes \Pi_m\right) \left(S_{vec}- \mu \right)\right).$$
Moreover, $A_m$ is an orthogonal projection matrix and we have the following equalities  
$$\left\Vert A_m \left(y - f\right) \right\Vert^2 = n \left\Vert\left(\Pi_m \otimes \Pi_m\right) \left(S_{vec}- \mu \right) \right\Vert^2=\left(y-f\right)^\top A_m  \left(y-f \right),$$
$$\mathrm{Tr}\left( A_m \right) = \frac{n}{n}\mathrm{Tr}\left(\Pi_m \otimes \Pi_m \right)= D_m,$$
$$ \mathrm{Tr} \left(A_m \left( Id_{n} \otimes \Phi \right) \right) =  \frac{n}{n}\mathrm{Tr}\left(\left(\Pi_m \otimes \Pi_m \right) \Phi \right) = \mathrm{Tr}\left(\left(\Pi_m \otimes \Pi_m \right) \Phi \right) .$$

Now we can use Proposition \ref{proBaraud} with $\widetilde{A}= A_m$, $\varepsilon_i = y_i-\mu$, $\mathrm{Tr}\left(A_m \right)= D_m$, $\rho \left( A_m \right)= 1$, $\delta^2=\delta_m^2$ and $\beta \geqslant 2$.

This gives for all $x >0$  \\
$\mathbb{P} \left( \left(y-f\right)^\top A_m  \left(y-f \right) \geqslant \mathrm{Tr}\left( \left(\Pi_m \otimes \Pi_m \right)\Phi \right) \left[1 + \sqrt{\frac{x}{D_m}}\right]^2 \right) \leqslant C_2(\beta) \frac{\E \left[ \left\Vert y_1 - \mu \right\Vert^\beta \right]D_m^{\frac{\beta}{2} + 1}}{\mathrm{Tr}\left( \left(\Pi_m \otimes \Pi_m \right)\Phi \right)^\frac{\beta}{2} x^\frac{\beta}{2}},$ \\

that is \\

$\mathbb{P} \left(  \left\Vert\left(\Pi_m \otimes \Pi_m\right) \left(S_{vec}- \mu \right) \right\Vert^2 \geqslant\frac{1}{n} \mathrm{Tr}\left( \left(\Pi_m \otimes \Pi_m \right)\Phi \right) \left[1 + \sqrt{\frac{x}{D_m}}\right]^2 \right) \leqslant C_2(\beta) \frac{\E \left[ \left\Vert y_1 - \mu \right\Vert^\beta \right]D_m^{\frac{\beta}{2} + 1}}{\mathrm{Tr}\left( \left(\Pi_m \otimes \Pi_m \right)\Phi \right)^\frac{\beta}{2} x^\frac{\beta}{2}}$. \vskip .1in
In order to use this deviation bound to  obtain the inequalities 

$$\mathbb{P} \left(\left\Vert \left( {\Pi }_{m}\otimes 
{\Pi }_{m} \right)\left(\mu -S_{vec}\right) \right\Vert^2 \geqslant  \frac{ \alpha}{4} \mathrm{Tr}\left( \left( {\Pi }_{m}\otimes 
{\Pi }_{m}\right){{\Phi} }\right) \right)\leqslant \tilde{C} \frac{1}{n^{\gamma}}$$
and
$$\mathbb{P} \left(\left\Vert \left( {\Pi }_{m}\otimes 
{\Pi }_{m} \right)\left(\mu -S_{vec} \right)\right\Vert \geqslant  \frac{ \alpha}{8\left\Vert \mu  \right\Vert} \mathrm{Tr}\left( \left( {\Pi }_{m}\otimes 
{\Pi }_{m}\right){{\Phi} }\right) \right)\leqslant \tilde{C} \frac{1}{n^{\gamma}}$$
with $\gamma\geqslant \frac{q}{1-2q/\kappa}$,
we need to find $x>0$ satisfying the three following facts
 \begin{equation}
\label{nec1}
\forall m \in \mathcal{M} \quad 
\frac{\alpha}{4}\geqslant \frac{1}{n}\left( 1 + \sqrt{\frac{x}{D_m}}\right)^2
\end{equation}
\begin{equation}
\label{nec2}
\forall m \in \mathcal{M} \quad \left(\frac{\alpha}{ 8\left\Vert \mu  \right\Vert} \right)^2 \mathrm{Tr}\left( \left( {\Pi }_{m}\otimes 
{\Pi }_{m}\right){{\Phi} }\right) \geqslant 
\left(\frac{\alpha}{ 8\left\Vert \mu  \right\Vert} \right)^2C_{inf}\geqslant \frac{1}{n}\left( 1 + \sqrt{\frac{x}{D_m}}\right)^2
\end{equation}
\begin{equation}
 \label{nec3}
\frac{ D_m^{\frac{\beta}{2} + 1}}{\mathrm{Tr}\left( \left(\Pi_m \otimes \Pi_m \right)\Phi \right)^\frac{\beta}{2} x^\frac{\beta}{2}} = \frac{ D_m}{\delta_m^\beta x^\frac{\beta}{2} }\leqslant C\frac{1}{n^{\gamma}}.
 \end{equation}

(\ref{nec1}) and \eqref{nec2} hold for the choice  $x=D_m n^r$ with $r<1$ and if $n$ is large enough to have 
\begin{equation}
\label{nlarge1}
\frac{1}{n} \left(1 + \sqrt{{n^r}}\right)^2\leqslant \frac{\alpha}{4}
\end{equation}
and
\begin{equation}
\label{nlarge2}
\frac{1}{n} \left(1 + \sqrt{{n^r}}\right)^2 \leqslant \left(\frac{\alpha}{ 8\left\Vert \mu  \right\Vert} \right)^2 C_{inf}
\end{equation}
\\


In order to obtain (\ref{nec3}) with $x=D_m n^r$, we use the inequality $D_m \leqslant n$  which gives  
$$ \frac{ D_m}{\delta_m^\beta x^\frac{\beta}{2} } \leqslant \frac{1}{\delta_{m}^\beta D_m^\frac{\beta}{2} n^{r\beta/2-1}}.$$

Moreover  
$$\E \left[ \left\Vert y_1 - \mu \right\Vert^\beta \right] \leqslant \E \left[ \left(\left\Vert y_1\right\Vert +\left\Vert \mu \right\Vert\right)^\beta \right],$$
and by using properties of convexity we obtain 
$$\E \left[ \left\Vert y_1 - \mu \right\Vert^\beta \right] \leqslant 2 ^ {\beta  - 1} \left(\E \left[ \left\Vert y_1 \right\Vert^\beta \right]  + \left\Vert \mu \right\Vert^\beta \right).$$
With the Jensen's inequality we get:
$$\E \left[ \left\Vert y_1 - \mu \right\Vert^\beta \right] \leqslant 2 ^ {\beta } \E \left[ \left\Vert y_1 \right\Vert^\beta \right].$$
In conclusion, with $r= \frac{\frac{\kappa}{2} + 2}{\beta }<1$  we obtain for $n\geqslant n(\kappa,\beta,\alpha,C_{inf},\Sigma)$ 
\begin{equation}
\label{eqA2}
Q2 \leqslant  2 ^ {\beta  + 1}C_2(\beta) \frac{\E \left[ \left\Vert x x^\top  \right\Vert^\beta \right]}{\delta_{m}^\beta D_m^\frac{\beta}{2} }\frac{1}{n^{\kappa/4}}
\end{equation}
where $C_2\left(\beta \right)$ is the constant which appears in Proposition \ref{proBaraud}. 

In conclusion, combining \eqref{eqA1} and \eqref{eqA2}
$$\mathbb{P}\left(\widehat{\delta}_m^2 \leqslant \left(1 - \alpha \right)\delta_m^2 \right) \leqslant \frac{1}{n^{\kappa/4}} \left(
C_2(\beta) 2 ^ {\beta  + 1} +  C_1 \left( \beta \right) \frac{1 }{\alpha^\frac{\beta}{2} }\right)\E \left[ \left\Vert x x^\top  \right\Vert^\beta \right] \delta_m^{-\beta} D_m^{-\frac{\beta}{2}}$$
for $n\geqslant n(\kappa,\beta,\alpha,C_{inf},\Sigma)$ .

To conclude, remark that $\frac{\kappa}{4} \left(1-2q/\kappa\right)=\frac{\kappa-2q}{4} > \frac{2+2q}{4} \geqslant q$ as $q\leqslant 1$.

\end{proof}

\begin{proof}{\textbf{Lemma~\ref{evident}}}\\
Recall that $\Pi_m \otimes \Pi_m$ is an orthogonal projection matrix. Hence there exists an orthogonal matrix $P_m$ such that $P_m ^{\top}\left( \Pi_m \otimes \Pi_m\right) P_m = D$, with $D$ a diagonal matrix with $D_{ii}= 1$ if $i \leqslant D_m$, and $D_{ii } = 0$ otherwise. Then if $\Psi$ is symmetric non-negative definite we have :
$$\mathrm{Tr}\left( \left( {\Pi }_{m}\otimes 
{\Pi }_{m}\right) \Psi \right) = \mathrm{Tr}\left(D P_m^{\top}\Psi P_m \right)$$
$$ =\sum_{l=1}^{p^2}\sum_{k=1}^{p^2}D_{kl} \left( P_m^{\top}\Psi P_m  \right)_{kl}=\sum_{l=1}^{p^2}D_{ll} \left( P_m^{\top}\Psi P_m  \right)_{ll}$$
$$=\sum_{l=1}^{D_m}\left( P_m^{\top}\Psi P_m  \right)_{ll} \in \left[ 0; \mathrm{Tr}\left(\Psi \right)\right].$$
Indeed, $P_m^{\top}\Psi P_m $ is non-negative definite so all its diagonal entries are non-negative.
\end{proof}

\bibliographystyle{abbrv}
\bibliography{bibli_chili}

\end{document}